\documentclass[a4paper]{article}

\usepackage{amsmath,amsthm,amsfonts}
\usepackage[english]{babel}
\usepackage[utf8]{inputenc}
\usepackage{amsmath}
\usepackage{graphicx}
\usepackage{color}

\title{Loop conditions}
\author{Miroslav Ol\v s\'ak} 
\date{\today}

\newcommand{\var}[1]{\mathcal{#1}}
\newcommand{\alg}[1]{\mathbf{#1}}
\newcommand{\lc}[1]{#1}
\newcommand{\relstr}[1]{\mathbb{#1}}
\newcommand{\tuple}[1]{\mathbf{#1}}

\newtheorem{theorem}{Theorem}
\newtheorem{definition}{Definition}
\newtheorem{proposition}{Proposition}

\newtheorem{claim}{Claim}

\newtheorem{corollary}{Corollary}

\usepackage{todonotes}

\begin{document}
\maketitle

\begin{abstract}
We prove that the existence of a term $s$ satisfying $s(x,y,y,z,z,x)=s(y,x,z,y,x,z)$ is the weakest non-trivial strong Maltsev condition given by a single identity. 
\end{abstract}

\section*{Acknowledgement}

This research was supported by the Czech Science Foundation (GAČR), grant 13-01832S.

\section{Introduction}

One of the most interesting by-products of the research on the fixed-template constraint satisfaction problems is the result of Mark Siggers~\cite{Siggers}, which says that there exists a weakest non-trivial strong Maltsev condition for idempotent locally finite varieties.

Let us recall that a \emph{strong Maltsev condition} (see \cite{BS,MMT,Bergman}) is a condition for a variety (or an algebra) postulating the existence of finitely many terms satisfying a given finite set of identities. Such conditions can be compared by their strength: a condition $C$ is \emph{weaker} than $D$ if each variety satisfying $D$ also satisfies $C$. The weakest conditions are those satisfied in every variety, we call them \emph{trivial}. The concept of strength can be naturally relativized to special types of varieties, such as idempotent locally finite varieties in the above mentioned result of Siggers.

His weakest non-trivial condition  has an especially simple form: it is given by a single linear identity in one operation symbol appearing on both sides, namely 
$$s(x,y,y,z,z,x)=s(y,x,z,y,x,z). $$
A well known fact, stated here as Proposition~\ref{loop-eq}, is that
strong Maltsev conditions of this form can be characterized by the existence of a loop
in certain binary relations compatible with algebras in the variety. 
This structural property of compatible relations proved useful (see~\cite{Cyclic,AbsSurvey}) and inspired the following terminology.

\begin{definition}
A \emph{loop condition} is a strong Maltsev condition given by a single identity of the form $$t(x_1, \dots, x_n) = t(y_1, \dots, y_n),$$
where $x_1, \dots, x_n, y_1, \dots, y_n$ are variables.
\end{definition}

In idempotent locally finite varieties, the 6-ary loop condition by Siggers is equivalent to many other loop conditions, such as the existence of a 4-ary term $t$ satisfying $t(r,a,r,e) = t(a,r,e,a)$~\cite{OptimalStrong}.

What happens when we drop the local finiteness assumption?
While there is no weakest non-trivial strong Maltsev condition for general varieties~\cite{Taylor88,Noncovering}, it turned out that there is one for idempotent varieties~\cite{Olsak}; for instance, the existence of a $6$-ary $t$ satisfying $t(x,y,y,y,x,x)=t(y,x,y,x,y,x)=t(y,y,x,x,x,y)$.
Such a condition cannot be a loop condition: A. Kazda~\cite{Kazda} proved that the free idempotent algebra over $\{x,y\}$ in the signature consisting of one ternary operation symbol $w$ modulo the weak near unanimity identities $w(x,y,y)=w(y,x,y)=w(y,y,x)$ does not satisfy any non-trivial loop condition. 

On the other hand, a consequence of the main result of this paper, Theorem~\ref{sym-class}, is that the existence of a 6-ary Siggers term is a weakest non-trivial loop condition \emph{in general}, that is, for varieties that are not necessarily idempotent or locally finite. To describe the main result, it will be convenient to assign a directed graph (digraph) to each loop condition in a natural way:

\begin{definition}
  \label{GC-def}
  Let $\lc C$ be a loop condition given by an identity
  $t(x_1, \dots, x_n)=t(y_1, \dots, y_n)$, where
  $x_1, \dots, x_n, y_1, \dots, y_n$ are variables from a set $V$.
  The \emph{digraph of $\lc C$}, denoted $\relstr G_{\lc C}$, is the
  digraph $(V,E)$ with vertex set $V$ equals to the variable set and edge
  set $E=\{(x_i, y_i): i=1, \dots, n\}$.
\end{definition}

For example, the digraph of the loop condition $s(x,y)=s(y,x)$ is $\relstr{K}_2$ and 
the digraph of the $6$-ary loop condition by Siggers is $\relstr{K}_3$, where $\relstr{K}_i$ denotes the complete loopless digraph on $i$ vertices.

It is easy to see that two loop conditions, whose associated digraphs are isomorphic, are equivalent. Therefore, we may talk about ``loop condition $\relstr{G}$'' (or sometimes ``$\relstr{G}$ loop condition'') instead of ``a loop condition whose associated digraph is $\relstr{G}$''.  Also note that $\relstr{G}_{\lc C}$ contains a loop if and only if $\lc C$ is trivial, i.e., satisfied in every algebra.

Our main result, Theorem~\ref{sym-class}, fully classifies the stregth of undirected loop conditions: Each such nontrivial condition is either equivalent to the existence of a commutative term (condition $\relstr{K}_2$) or the existence of a Siggers $6$-ary term (condition $\relstr{K}_3$).

\medskip
\begin{figure}[!ht]
  \def\svgwidth{2cm}
  \centering
     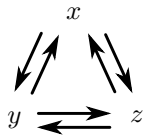
  \caption{The digraph of the Siggers term.}
\end{figure}

\section{Preliminaries}


An $n$-ary operation $f$ on a set $A$ is \emph{compatible} with an $m$-ary relation $R \subseteq A^m$, or $R$ is \emph{compatible} with $f$, if $f(\tuple{r}_1, \ldots, \tuple{r}_n) \in R$ for any $\tuple{r}_1, \ldots, \tuple{r}_n \in R$. Here (and later as well) we abuse the notation and use $f$ also for the $n$-ary operation on $A^m$ defined from $f$ coordinate-wise.

An algebra $\alg A=(A, f_1, f_2, \ldots)$ is said to be compatible with a relational structure $\relstr A=(A, R_1, R_2, \ldots)$ if all the operations $f_1, f_2, \ldots$ are compatible with all the relations $R_1, R_2, \ldots$.

\emph{Digraph} is a relational structure $\relstr G=(G, E)$ with one binary relation $E$. If $E$ is symmetric, then the digraph $\relstr G$ is called \emph{undirected}. A \emph{loop} in $\relstr G$ is a pair of the form $(x,x)\in E$.

We will extensively use a standard method for building compatible
relations from existing ones -- primitive positive (\emph{pp}, for short) definitions.
A relation $R$ is {\it pp-definable\/} from relations
$R_1,\ldots,R_n$ if it can be defined by a first order formula using variables,
existential quantifiers, conjunctions, the equality relations, and predicates $R_1,\ldots,R_n$.
Clauses in pp-definitions are also referred to as \emph{constraints}.
Recall that $R_1,\ldots,R_n$ are compatible with an algebra, then so is $R$.

It is sometimes helpful to visualize a pp-definition of $k$-ary relation $R$ from a digraph $\relstr G$ as a digraph $\relstr{H}$ with $k$ distinguished vertices $u_1, \ldots, u_k$:
the vertices of $\relstr{H}$ are the variables, its edges correspond to the constraints, and the distinguished vertices are the free variables.
Observe that $(v_1, \ldots, v_k)$ is in $R$ if and only if there is a digraph homomorphism $\relstr{H} \to \relstr{G}$ which maps $u_i$ to $v_i$ for each $i = 1, \ldots, k$. In a similar way, we can visualize pp-definitions from a relational structure consisting of more than one binary relation. 

A \emph{pp-power} of a relational structure $\relstr A$ on $A$ is a relational structure $\relstr{B}$ on $A^l$ whose relations are pp-definable from $\relstr A$ in the sense that a $k$-ary relation from $\relstr{B}$, regarded as a $(k \cdot l)$-ary relation on $A$, is pp-definable from $\relstr A$. Observe that if $\relstr{A}$ is compatible with an algebra $\alg A$, then the pp-power $\relstr{B}$ is compatible with the algebraic power $\alg A^l$.


For more background on pp-definitions and its relevance for constraint satisfaction problems we refer the reader to \cite{PolSurvey}.

Finally, we state the promised correspondence between loop conditions and loops in digraphs.






\begin{proposition}
\label{loop-eq}
Let $\var V$ be a variety and $\lc C$ a loop condition.
The following conditions are equivalent.
\begin{enumerate}
\item[(i)] $\var V$ satisfies $\lc C$.
\item[(ii)] For every $\alg A\in\var V$ and for every digraph
  $\relstr G = (A, G)$ compatible
  with $\alg A$, the following holds: If there is a digraph homomorphism
  $\relstr G_{\lc C}\to \relstr G$ then $\relstr G$ contain a loop.
\item[(iii)] For every $\alg A\in\var V$ and for every digraph
  $\relstr G = (A, G)$ compatible
  with $\alg A$, the following holds: If a subdigraph of $\relstr G$ is
  isomorphic to $G_{\lc C}$ then $G$ contains a loop.
\end{enumerate}
\end{proposition}
\begin{proof}
  (i)$\Rightarrow$(ii)
  Let $\lc C$ be of the form
  $t(x_1, \ldots, x_n) = t(y_1, \dots, y_n)$ and
  $f\colon \relstr G_{\lc C}\to\relstr G$ be a digraph
  homomorphism. Then $(f(x_i), f(y_i))\in G$ for all
  $i=1,\ldots, n$. By compatibility of $\relstr G$ with $\alg A$ also
  $$(t(f(x_1), \ldots, f(x_n)), t(f(y_1), \dots, f(y_n)))\in G,$$
  Since equality
  $$t(f(x_1), \ldots, f(x_n)) = t(f(y_1), \dots, f(y_n))$$
  is ensured by $\lc C$, we get the desired loop in $\relstr G$. 

(ii)$\Rightarrow$(iii) Trivial.

(iii)$\Rightarrow$(i) Let $\alg F$ be the free algebra in $\var V$
generated by the vertices of $\relstr G_{\lc C}$ and let $E$ be the subuniverse of
$\alg F^2$ generated by the edges of $\relstr G_{\lc C}$.
The set $E$ is the edge--set a
digraph compatible with $\alg F$ containing $\relstr G_{\lc C}$ as a subgraph. Therefore, by
(iii), there is loop $(a,a)$ in $E$. The pair $(a,a)$ is generated from edges of $\relstr G_C$, so there is a term operation $t$ of $\alg F^2$ taking edges of
$\relstr G_{\lc C}$ as arguments and returning $(a,a)$. Thus, in $\alg F$, the
term $t$ satisfies the loop condition $C$, at least when we plug in the generators. By the universality of free algebra, the term $t$
satisfies $\lc C$ in general.
\end{proof}

\begin{corollary}
\label{graph-homo}
Let $\lc C$, $\lc D$ be loop conditions. If there is a digraph
homomorphism $\relstr G_{\lc C}\to \relstr G_{\lc D}$, then $\lc C$ implies
$\lc D$. In particular, $\lc C$ implies $\lc D$ whenever $\relstr G_{\lc C}$ is a subdigraph of
$\relstr G_{\lc D}$.
\end{corollary}

\begin{corollary}
  \label{graph-eq}
  Loop conditions with homomorphically equivalent digraphs are
  equivalent. In particular, loop conditions with isomorphic digraphs
  are equivalent.
\end{corollary}

%
%
%
%

\section{Undirected case}

In this section, we will focus on loop conditions with 
undirected digraphs, briefly graphs.


We start with several simple consequences of 
Corollary~\ref{graph-homo}.


\begin{proposition}
  \label{triv-sym-impl}
  \begin{enumerate}
    \item Any loop condition of a bipartite graph is equivalent to the edge loop condition (commutativity).
    \item The $(l+2)$-cycle loop condition implies the $l$-cycle loop condition for any odd length $l\geq3$.
    \item The $n$-clique loop condition implies the $(n+1)$-clique loop condition for any size $n\geq3$.
    \item For any non-bipartite graph $\relstr G$ there is an odd
      length $l\geq 3$ such that the $l$-cycle loop condition implies
      the loop condition given by $\relstr G$.
    \item For any loopless digraph $\relstr G$ there is a size
      $n\geq 3$ such that the loop condition given by $\relstr G$
      implies the $n$-clique loop condition.
  \end{enumerate}
\end{proposition}

\medskip
\begin{figure}[!ht]
  \def\svgwidth{10cm}
  \centering
     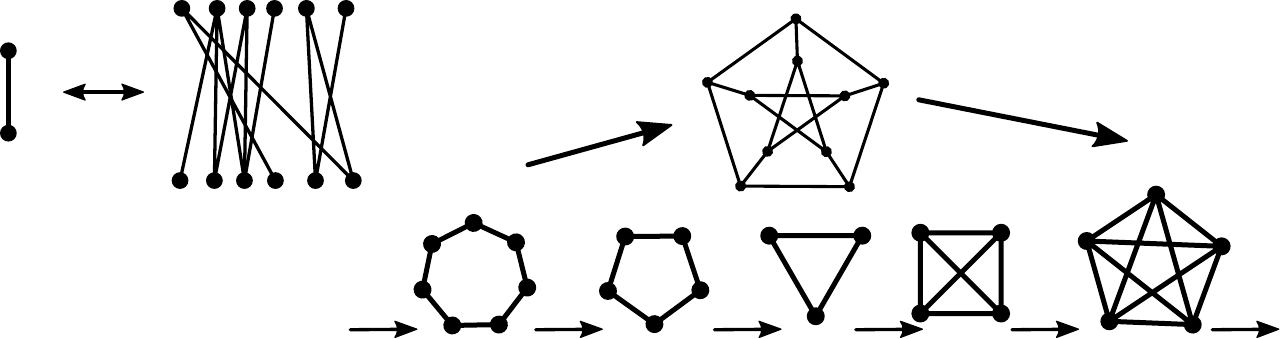
  \caption{Scheme of the easy implications between
  undirected loop conditions given by Proposition~\ref{triv-sym-impl}.}
\end{figure}

Our aim in the rest of the chapter is to reverse the second and the third implication, that is,
$(2n+3)$-cycle to $(2n+1)$-cycle and $(n+2)$-clique to $(n+3)$-clique,
where $n\geq 1$. It will follow that all the loopless
non-bipartite graphs are equivalent as loop conditions and that they
are weakest among all non-trivial loop conditions.

\begin{proposition}
\label{larger_cycles}
The $l$-cycle loop condition implies the $(l+2)$-cycle loop
condition for any odd length $l\geq 3$.
\end{proposition}
\begin{proof}
There is a graph homomorphism from the
$l^2$-cycle to the $(l+2)$-cycle because both cycles are odd and
$l^2\geq l+2$.
It is thus sufficient to show that the $l$-cycle loop condition
implies the $l^2$-cycle one.

We use Proposition~\ref{loop-eq}. Let $\alg A$ be an
algebra on a set $A$ such that every digraph compatible with $\alg A$ containing a
homomorphic image of the $l$-cycle has a loop. We need to prove that every
compatible digraph $\relstr G=(A, G)$ containing a homomorphic image of the
$l^2$-cycle has a loop. Take such a digraph $\relstr G$ where the cycle is
formed by vertices $v_0, v_1, \ldots, v_{l^2-1}$. We may assume that $G$ is symmetric, since the set of symmetric edges of $\relstr{G}$ is clearly pp-definable from $\relstr G$.

\medskip
\begin{figure}[!ht]
  \def\svgwidth{8cm}
  \centering
     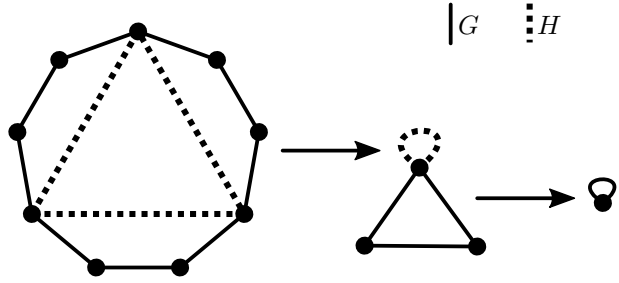
  \caption{Getting a loop from 9-cycle using the triangle
  loop condition.}
\end{figure}

We construct a binary relation $H$ on $A$: vertices $x$ and $y$ are by
definition $H$-adjacent if there is a $G$-walk from $x$ to $y$ of
length $l$. Since $H$ is pp-defined from a compatible relation, it is
compatible with $\alg A$ too.
Vertices $v_0, v_l, v_{2l}, \ldots, v_{l(l-1)}$ form a
cycle in $H$ of length $l$. By the assumption on algebra $\alg A$,
there is a loop in $H$. A loop in $H$ is a homomorphic image of an
$l$-cycle in $G$. Using the assumption on $\alg A$ again we
get a loop in $G$.
\end{proof}

\begin{proposition}
\label{smaller_cliques}
The $(n+1)$-clique loop condition implies the $n$-clique loop
condition for any $n\geq 3$.
\end{proposition}
\begin{proof}
As in the previous proposition, we use Proposition~\ref{loop-eq}. Let $\alg A$ on a set $A$
be an algebra satisfying the $(n+1)$-clique loop condition.
Let $\relstr G=(A, G)$ be a digraph
compatible with $\alg A$ containing an $n$-clique
$a_1, a_2, \ldots, a_n$ as a subgraph.
It suffices to prove that $\relstr G$ has to have a loop. 
As before we may assume that $\relstr{G}$ is a graph.

Let us pp-define a 4-ary relation $R$ on $A$ as
follows. $R(u,v,x,y)$ if and
only if there are elements $x_1, \ldots, x_{n-2}, w\in A$ such that
all the vertices $x_i$ are pairwise $G$-adjacent to each other, they are also
$G$-adjacent to vertices $x, y, v, w$ and moreover
$G(u,w), G(w,x), G(v,y)$.
From $R$ we pp-define a binary relation
$F(x,y)\Leftrightarrow\exists u\in A\colon R(u,u,x,y)$.

\medskip
\begin{figure}[!ht]
  \def\svgwidth{8cm}
  \centering
     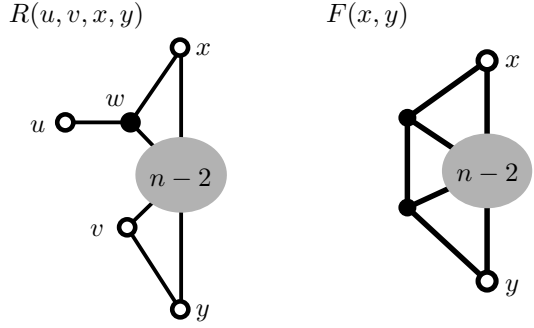
  \caption{A vizualisation of the definitions of $R$ and $F$.}
\end{figure}

Observe that the pp-definition of $F$ is symmetric, so $F$ itself is
symmetric. We regard $F$ as the set of edges of a graph $\relstr F=(A,F)$ compatible with $\alg A$.

\begin{claim}
  \label{R-description}
  Let $u,v,x,y\in \{a_1, \ldots, a_n\}$. Then $R(u,v,x,y)$ whenever one of the following condition are met:
  \begin{enumerate}
  \item[(a)] $u\neq v$ and $x=y\neq v$.
  \item[(b)] $u=v$ and $x=u=v$, $y\neq x$.
  \end{enumerate}
\end{claim}
To prove the claim, we need to find correct values of variables
$x_1,\ldots, x_{n-2}, w$ in the definition of $R$ to meet the constraints
of $R$. We do it separately for the two cases.
 \begin{enumerate}
 \item[(a)] We set $w=v$ and variables $x_i$ to vertices $a_i$
   different from $v$, $x$.
 \item[(b)] We set $w=y$ and variables $x_i$ to vertices $a_i$
   different from $x$, $y$.
 \end{enumerate}

The next claim immediately follows from the case (b) of the previous one.

\begin{claim}
  \label{F-description}
  If $x,y\in \{a_1, \ldots, a_n\}$ and $x\neq y$, then $F(x,y)$.
\end{claim}

Finally we define a digraph $\relstr Q=(A^2, Q)$, where
the binary relation $Q$ is defined as follows:
$Q((u_1,u_2),(v_1,v_2))$ if and only if there are
$x_1,x_2,\ldots,x_{n+1}$ such that every pair of different indices
$i_1,i_2$, with the possible exception of pairs $\{1,2\}$ and $\{3,4\}$, satisfies
$F(x_{i_1},x_{i_2})$ and, moreover, $R(u_1,v_1,x_1,x_2)$ and $R(u_2,v_2,x_3,x_4)$.

\medskip
\begin{figure}[!ht]
  \def\svgwidth{9cm}
  \centering
     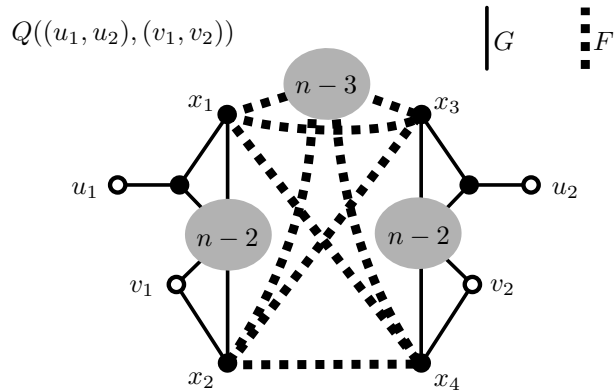
  \caption{A visualization of the definitions of $Q$.}
\end{figure}

\begin{claim}
  If $u_1, u_2, v_1, v_2\in \{a_1, \ldots, a_{n+1}\}$ and
  $(u_1,u_1) \neq (v_2,v_2)$, then $Q((u_1,u_1), (v_2,v_2))$.
\end{claim}
  To prove the claim, let variables $x_1, \ldots, x_{n+1}$ be as in the
  definition of $Q$. We satisfy the constraints of $Q$ by means of 
  Claims~\ref{R-description},~\ref{F-description}. Because of $F$ clauses in the definition of $Q$, all 
  pairs of variables from $\{x_1, \ldots, x_{n+1}\}$ should differ, with two possible exceptions
  $x_1=x_2$ and $x_3=x_4$.
  We analyse cases according to equalites among $u_1,u_2,v_1,v_2$. In each
  case we use Claim~\ref{R-description} and assign suitable values
  to $x_1,x_2,x_3,x_4$.
  \begin{enumerate}
  \item[(a)] $u_1\neq v_1$ and $u_2\neq v_2$: We want to use~(a), that is, we want to have $x_1=x_2\neq v_1$ and $x_3=x_4\neq v_2$. First we
    choose $x_1=x_2$, then $x_3=x_4$ different from $v_2$ and
    $x_1$. This is possible since $n\geq 3$.
  \item[(b)] $u_1=v_1$ and $u_2\neq v_2$: We want to satisfy $x_1=u_1$,
    $y_1\neq x_1$ (and use (b)), and
    $x_3=x_4\neq v_2$ (and use (a)). So we put $x_1=u_1$, then choose a
    value for $x_3=x_4$ different from $x_1, v_2$ and, finally, we choose
    a value for $x_2$ different from $x_1, x_3$.
  \item[(c)] $u_1\neq v_1$ and $u_2=v_2$ is analogous to the previous case.
  \item[(d)] $u_1=v_1$ and $u_2=v_2$ can not happen since
    $(u_1,u_1) \neq (v_2,v_2)$.
  \end{enumerate}
  In every case $x_1=x_2$ or $x_3=x_4$, therefore the remaining variables $x_i$
  can be completed so that all the constraints of $Q$ are satisfied.

Let us finish the proof. The digraph $\relstr Q$ is compatible with the algebra
$\alg A^2$ since it is a pp-power of $\relstr G$. Moreover, $\relstr Q$
contains a clique of size $n^2\geq n+1$. The $n$-clique loop condition
holds also for $\alg A^2$, so there is a loop in $\relstr Q$.

The loop in $\relstr Q$ is represented by elements $x_1,\ldots,x_{n+1}$ in
$A$ such that $F(x_i,x_j)$ whenever $i\neq j$. Since $\relstr F$ is pp-defined
from $\relstr G$, it is compatible with $\alg A$. Therefore, there is
a loop in $\relstr F$.

Finally, a loop in $\relstr F$ yields a $(n+1)$-clique in $\relstr G$ and,
consequently, the sought after loop in $\relstr G$.
\end{proof}

\begin{theorem}
  \label{sym-class}
There are exactly three equivalence classes of loop conditions given by  undirected digraphs $\relstr G$: 
\begin{enumerate}
\item[(1)] loop conditions $\relstr G$, where $\relstr G$ is bipartite,
\item[(2)] loop conditions $\relstr G$, where $\relstr G$ is non-bipartite and loopless,
\item[(3)] loop conditions $\relstr G$, where $\relstr G$ contains a loop (trivial).
\end{enumerate}
Conditions (1) imply (2) imply (3).
Conditions (2) are the weakest non-trivial loop conditions.
\end{theorem}

\begin{proof}
If a graph $\relstr G$ is bipartite (contains a loop, respectively), then the $\relstr{G}$ loop condition is equivalent to the edge loop condition by item 1 of Proposition~\ref{triv-sym-impl} (is trivial, respectively). If $\relstr{G}$ is non-bipartite and loopless, then the $\relstr{G}$ loop condition implies a clique loop condition (item 5 of Proposition~\ref{triv-sym-impl}), which implies the triangle loop condition (by Proposition~\ref{smaller_cliques}), which  implies the $l$-cycle loop condition for any odd $l$ (by Proposition~\ref{larger_cycles}), which, finally, implies the $\relstr{G}$ loop condition (by item 4 of Proposition~\ref{triv-sym-impl}).

Clearly, conditions (1) imply (2) imply (3), and (3) do not imply (2). The implication from (2) to (1)  cannot be reversed even for idempotent finitely generated varieties: an example of an algebra satisfying (2) but not (1) is the algebra $(\{0,1,2\}, m)$ where $m(x,y,z)=x+y-z$ modulo 3.
\end{proof}

Theorem~\ref{sym-class} also provides an alternative proof to the fact~\cite{Olsak} that the existence of a near unanimity
term implies the existence of a Siggers term. Recall that a \emph{near unanimity} (NU) term is an $n$-ary term $t$ satisfying the identity
$t(x,\ldots,x,y,x,\ldots,x)= x$ for all positions of $y$.

\begin{theorem}
  If an algebra (or a variety) has an $n$-ary NU term, then it also has the Siggers term.
\end{theorem}
\begin{proof}
  Let $t$ be the NU term. The value $t(x_1, x_2, \ldots, x_n)$ may be
  expressed in the following two ways.
  \[
  \vbox{\def\c#1{\hfil#1\hfil}
  \halign{$#{}$&$t(t(\c{#},{}$&$\c{#},{}$&$\c{#},\ldots,{}$&$\c{#}),{}$&%
                  $t(\c{#},{}$&$\c{#},{}$&$\c{#},\ldots,{}$&$\c{#}),{}$&%
       $),\ldots,{}t(\c{#},\ldots,{}$&$\c{#},{}$&$\c{#},{}$&$\c{#}))$&${}#$\cr
 &x_2&x_1&x_1&x_1&  x_2&x_3&x_2&x_2&   x_n&  x_n   &  x_n  &x_1&\cr
=&x_1&x_2&x_3&x_n&  x_1&x_2&x_3&x_n&   x_1&x_{n-2}&x_{n-1}&x_n&.\hss\cr
  }}
  \]
  This identity may be interpreted as a loop identity (whose graph is the $n$-clique) for a composed $n^2$-ary term.
  Therefore the $n$-clique loop condition is satisfied and then also the
  triangle loop condition is satisfied since all the loop conditions given by non-bipartite loopless graphs are equivalent.  
\end{proof}

\section{Directed case}

A digraph $\relstr G$ is said to be {\it smooth\/} if every vertex has an incoming
and an outgoing edge; $\relstr G$ is said to have
{\it algebraic length 1} if there is no graph homomorphism from
$\relstr G$ to a directed cycle of length greater than one.
The following theorem holds for finite algebras~\cite{OptimalStrong}.

\begin{theorem}
  Let $\alg A$ be a finite algebra. Let $\relstr G, \relstr H$ be weakly connected smooth
  digraphs of algebraic length one. Then $\alg A$ satisfies loop
  condition $\relstr G$  if and only if it satisfies loop condition  $\relstr H$.
\end{theorem}

Less formally, all connected smooth loop conditions with algebraic length
1 are equivalent for finite algebras. This gives us a simpler weakest loop
condition than the triangle for finite algebras: $s(a,r,e,a)=s(r,a,r,e)$.

For general varieties it is no longer true that all such loop conditions are equivalent. However, very recently, we have shown that all loop conditions given by strongly connected graphs of algebraic length 1 are. These results will appear in a forthcoming paper.


%
%

\bibliographystyle{plain}
\bibliography{bib-file.bib}

\end{document}